\documentclass[a4paper,reqno,oneside,10pt]{amsart} 

\usepackage[frenchb]{babel}
\usepackage[latin1]{inputenc}
\usepackage{epsfig}
\usepackage{graphicx}
\usepackage{psfrag}

\usepackage{amsmath}
\usepackage{amssymb}
\usepackage{enumerate}
\usepackage{mathrsfs}
\newcommand{\C}{\mathbf{C}}
\newcommand{\R}{\mathbf{R}}

\newcommand{\Z}{\mathbf{Z}}

\newcommand{\D}{\mathbf{D}}

\newcommand{\Hyp}{\mathbf{H}}

\newcommand{\Sph}{\mathbf{S}}

\def\a{\alpha}

\def\g{\gamma}
\def\G{\Gamma}
\def\d{\delta}
\def\e{\varepsilon}
\def\l{\lambda}
\def\f{\varphi}

\newcommand{\sys}{\mathsf{sys}}

\newcommand{\area}{\mathsf{aire}}
\newcommand{\ri}{\mathsf{R}}
\newcommand{\Isom}{\mathsf{Isom}}

\newcommand{\PSL}{\mathsf{PSL}}

\newcommand{\Imaginary}{\mathsf{Im}}
\newcommand{\Real}{\mathsf{Re}}

\newcommand{\diff}{\mathrm{d}}
\newcommand{\teich}{\mathsf{Teich}}
\newcommand{\ML}{\mathcal{ML}}
\newcommand{\PML}{\mathcal{PML}}
\newcommand{\simple}{\mathcal{S}}

\newcommand{\mrm}{\mathrm}

\theoremstyle{plain}
\newtheorem{theorem}{Th\'eor\`eme}[section] 
\newtheorem*{theoremnonumber}{Th\'eor\`eme}

\newtheorem{corollary}[theorem]{Corollaire}

\newtheorem{proposition}[theorem]{Proposition}

\newtheorem{lemma}[theorem]{Lemme}

\newtheorem*{lemmedeschwarz}{Lemme de Schwarz}

\newtheorem*{lemmedecontraction}{Lemme de contraction}

\newtheorem{thm-defi}[theorem]{Théorème - Définition}
\newtheorem{pro-defi}[theorem]{Proposition - Définition}

\theoremstyle{definition}

\newtheorem{question}{Question}[section]

\newtheorem{remark}{Remarque}[section]
\newtheorem{remarks}{Remarques}[section]

\newtheorem*{notations}{Notations}

\title[]{Trois applications du lemme de Schwarz aux surfaces hyperboliques}
\subjclass[2000]{30F45 \and 30F60}
\keywords{Injectivity radius\and geodesics\and Schwarz lemma}
\date{Le \today}
\thanks{This work has been fully supported by FIRB 2010 (RBFR10GHHH$_{003}$).}

\begin{document}

\maketitle

\begin{flushleft} 
            \textbf{Matthieu Gendulphe}\\
  \begin{small}Dipartimento di Matematica Guido Castelnuovo\\
                      Sapienza università di Roma,
                      Piazzale Aldo Moro, 00185 Roma\\
                      Matthieu@Gendulphe.com\end{small}
\end{flushleft}

\renewcommand{\abstractname}{Abstract} 
\begin{abstract}

Through the Schwarz lemma, we provide a new point of view on three well-known results of the geometry of hyperbolic surfaces. The first result deal with the length of closed geodesics on hyperbolic surfaces with boundary (\cite{thurston,parlier,papadopoulos}). The two others give sharp lower bounds on two metric invariants: the length of the shortest non simple closed geodesic, and the radius of the biggest embedded hyperbolic disk (\cite{yamada}). We also discuss a question of Papadopoulos and Théret (\cite{papadopoulos}) about the length of arcs on surfaces with boundary.\par
 In a sequel \cite{rayon}, we use a generalization of the Schwarz lemma (\cite{yau}) to study the injectivity radius of surfaces with bounded curvature.
\end{abstract}

\section{Introduction} 

Dans cet article et sa suite \cite{rayon}, nous nous intéressons aux applications du lemme de Schwarz à la géométrie des surfaces.\par

\subsection{Contenu de l'article} 
Le lemme de Schwarz, interprété géométriquement par Pick, semble peu utilisé en géométrie hyperbolique. Pourtant, il permet de comparer à peu de frais des surfaces hyperboliques très différentes.% Par exemple, étant donnée une surface hyperbolique pointée $(X,p)$, il assure l'existence d'un difféomorphisme contractant entre une surface hyperbolique non compacte $Y$ et $X\setminus\{p\}$. À bien y réfléchir, il n'y a pas de preuve naturelle de ce résultat relevant uniquement de la géométrie hyperbolique.\par
\ Avec le lemme de Schwarz, nous allons revisiter trois théorèmes classiques de la géométrie des surfaces hyperboliques. Nous verrons comment ce point de vue nouveau produit des preuves simples, courtes et ne nécessitant que peu de calculs.\par

 Le premier théorème (Thurston \cite{thurston}, Parlier \cite{parlier}) affirme que l'on peut décroître les longueurs de toutes les géodésiques fermées d'une surface hyperbolique à bord, quitte à diminuer la longueur du bord. Nous améliorons ce théorème, puis nous discutons le problème analogue pour les arcs. Nous répondons ainsi à une question de Papadopoulos et Théret (\cite{papadopoulos}). Ceci occupe la partie~\ref{sec:contraction} de l'article.\par
  Les deux autres théorèmes (Yamada \cite{yamada}) donnent des bornes inférieures optimales sur deux invariants métriques des surfaces hyperboliques~: la longueur de la plus courte géodésique fermée non simple, et la borne supérieure du rayon d'injectivité. Ces résultats sont respectivement redémontrés dans les parties~\ref{sec:intersection} et \ref{sec:yamada}. Pour minorer la borne supérieure du rayon d'injectivité, nous nous ramenons aux surfaces hyperboliques à pointes, afin de travailler avec une région cuspidale. Nous faisons alors appel à un résultat de Seppälä et Sorvali (\cite{sorvali} et appendice~\ref{sec:cercles-isom}) sur l'aire des régions cuspidales.\par
  
\subsection{Suite} Dans \cite{rayon} nous établissons des inégalités optimales sur la borne supérieure du rayon d'injectivité des surfaces à courbure bornée $|K|\leq 1$. Le lemme de Schwarz, sous une forme due à Yau (\cite{yau}), tient à nouveau un rôle clé. \par

\subsection{Conventions} Sauf mention du contraire, une surface est supposée connexe et sans bord. Une métrique \emph{hyperbolique} est une métrique \emph{complète} à courbure constante $-1$. Le bord d'une surface hyperbolique est toujours supposé \emph{géodésique}.\par
 Soit $X$ une surface hyperbolique à bord non vide. Un \emph{arc} est une immersion d'un intervalle compact dans $X$ telle que les extrémités, et seulement les extrémités, sont envoyées sur le bord $\partial X$. Les homotopies, ou les isotopies, préservent chaque composante de bord mais ne la fixe pas nécessairement point-à-point. Chaque classe d'homotopie d'arcs admet un représentant de longueur minimale, il s'agit de l'unique représentant géodésique orthogonal au bord. Nous parlerons parfois d'un arc pour désigner sa classe d'homotopie, dans ce cas la longueur de l'arc est celle de son plus court représentant.\par
 Une application \emph{conforme} est une application préservant les angles non orientés. Nous la supposons à dérivée partout non nulle.
 
\subsection{Remerciements} Je remercie Christophe Bavard pour des échanges qui ont suscité ce travail.

%%%%%%%%%%%%%%%%%%%%%%%%%%%%%%
\section{Contraction des géodésiques, non contraction des arcs}\label{sec:contraction}
%%%%%%%%%%%%%%%%%%%%%%%%%%%%%%

\subsection{Contraction des géodésiques des surfaces à bord}\label{sec:geodesics}%%%%%%%%%%%%%
Commençons par rappeler la version de Pick du lemme de Schwarz~:

\begin{lemmedeschwarz}
Une  fonction holomorphe $f:\D\rightarrow\D$ est ou bien une isométrie hyperbolique, ou bien une application contractante au sens où
$\|\diff f_z\|<1$ en tout point $z$ de $\D$, la norme étant celle induite par la métrique hyperbolique. 
\end{lemmedeschwarz}

\begin{remark}
L'énoncé reste vrai pour une fonction anti-holomorphe.% (composer la fonction avec une isométrie renversant l'orientation).
\end{remark}

%\begin{proof}
%Quitte à faire agir $\Isom^+(\D)$ à la source et au but, nous supposons $z=0$ et $f(0)=0$. Ainsi, la fonction $g(z)=f(z)/z$ se prolonge en une fonction holomorphe en l'origine, prenant la valeur $g(0)=f'(0)$. Si $|g(0)|=1$ alors $g$ est constante par le principe du maximum, et nous avons $f(z)=e^{i\theta} z$ pour un certain $\theta\in\Sph^1$. Sinon nous avons $\|\diff f_0\|=|f'(0)|=|g(0)|<1$.
%\end{proof}
 
 La bijectivité caractérise les automorphismes parmi les applications holomorphes du disque dans lui-même, le lemme de Schwarz-Pick se reformule donc de la manière suivante~: \emph{une application holomorphe $f:\D\rightarrow \D$ est ou bien bijective et il s'agit d'une isométrie, ou bien contractante}. Le cas des applications conformes sera particulièrement utile par la suite.\par

\begin{corollary}
Une application conforme entre deux surfaces hyperboliques sans bord est ou bien surjective, ou bien contractante. Si elle est surjective, alors il s'agit d'un revêtement riemannien non singulier. 
\end{corollary}

\begin{remark} Une application conforme entre deux surfaces hyperboliques \emph{fermées} est surjective, car elle est ouverte et fermée.\end{remark}

\begin{proof} %Soit $f:X\rightarrow Y$ une application conforme entre deux surfaces hyperboliques sans bord, et soit $\tilde f:\D\rightarrow \D$ l'un de ses relevés. En appliquant le lemme de Schwarz-Pick à $\tilde f$, nous trouvons $\|\diff f_x \|=\|\diff \tilde f_{\tilde x} \| \leq 1$ en tout point $x$ de $X$. Ainsi $f$ est un difféomorphisme local \emph{propre}, et $f:X\rightarrow f(X)$ est un revêtement (théorème d'Ehresmann).\par
Soit $f:X\rightarrow Y$ une application conforme entre deux surfaces hyperboliques sans bord.
 Selon le lemme de Schwarz, cette application est ou bien une isométrie locale, ou bien contractante. Clairement, $f$ est une isométrie locale si et seulement si l'image réciproque de la métrique de $f(X)$ (induite par celle de $Y$) coïncide avec la métrique de $X$. Or la métrique de $X$ se caractérise comme l'unique métrique complète à courbure constante $-1$ dans sa classe conforme. Nous observons que $f$ est une isométrie locale si et seulement si $f(X)=Y$. La complétude est la propriété clé.\par
 Si $f$ est une isométrie locale surjective, alors il en va de même pour n'importe lequel de ses relevés $\tilde f:\D\rightarrow \D$. Comme ce relevé envoie une géodésique de $\D$ sur une géodésique de $\D$, il est nécessairement injectif. Donc $\tilde f$ est une isométrie de $\D$, et $f$ est un revêtement non singulier.
 \end{proof}  

 L'application conforme non surjective la plus simple est l'inclusion. Nous allons construire deux types d'inclusions, nous nous intéresserons surtout au deuxième.\par
 Soit $X$ une surface hyperbolique, et $\Sigma\subset S$ un ensemble discret de points. Nous désignons par $S$ la surface lisse sous-jacente à $X$. D'après le théorème de Poincaré-Koebe, il existe une unique structure hyperbolique $Y$ sur $S\setminus\Sigma$ telle que l'inclusion $Y\hookrightarrow X$ soit conforme, donc contractante. Notez qu'à chaque point de $\Sigma$ correspond une pointe de $Y$. Cette méthode, permettant de comparer une surface compacte avec une surface non compacte, a été introduite par R.~Brooks (\cite{brooks}), et reprise par F.~Balacheff, E.~Makover et H.~Parlier (\cite{balacheff}).\par 
  Nous allons maintenant construire des inclusions entre des surfaces hyperboliques d'aires infinies suivant une idée qui remonte au moins à L.~Bers (\cite{bers}). Cette construction repose sur la correspondance entre les structures conformes et les métriques hyperboliques sur une surface donnée, donc sur le théorème d'uniformisation de Poincaré-Koebe.\par
 
 On appelle \emph{vasque} (\emph{funnel} en anglais) une partie d'une surface hyperbolique isométrique au quotient $\{z\in\Hyp~;~\Real(z)>0\}/\langle z\mapsto \a z \rangle$ pour un certain nombre réel $\a>1$. Il s'agit d'un bout d'aire infinie bordé par une géodésique fermée simple de longueur $\ln(\a)$. Une vasque est homéomorphe au cylindre ouvert $\Sph^1\times ]0,+\infty[$, et d'un point de vue conforme est équivalente à une couronne de module $\pi^2/\ln(\a)$. En particulier, une vasque admet une compactification conforme canonique, qui s'identifie au quotient $\{z\in\Hyp~;~\Real(z)\geq 0\}\cup\R_+^\ast/\langle z\mapsto \a z \rangle$.\par
 Soit $\widehat X$ une surface hyperbolique sans bord ayant au moins une vasque. En ajoutant une petite couronne à l'une des vasques, nous obtenons une surface $\widehat Y$ munie d'une structure conforme pour laquelle l'inclusion $f:\widehat X\hookrightarrow \widehat Y$ est conforme. En vertu du corollaire ci-dessus, l'inclusion $f$ est contractante relativement aux métriques de Poincaré sur $\widehat X$ et $\widehat Y$. Considérons le facteur conforme $\phi:\widehat X\rightarrow \R_+^\ast$ défini par  $\phi(x)=\|\diff f_x\|$ en tout point $x$ de $\widehat X$. Comme $f$ est contractante, nous avons $0< \phi <1$ sur $\widehat X$.  En particulier, pour tout sous-ensemble compact $K\subset \widehat X$, il existe $0<\l<1$ tel que $0<\phi<\l<1$ sur $K$.\par
 Nous appelons \emph{c\oe ur convexe} de $\widehat X$ la plus petite sous-surface convexe contenant toutes les géodésiques fermées de $\widehat X$. Il s'agit simplement du complémentaire des vasques dans $\widehat X$. Lorsque le c\oe ur convexe est d'aire finie, nous pouvons prendre pour $K$ un compact  contenant toutes les géodésiques fermés simples de $\widehat X$ (voir le corollaire~\ref{cor:simple}), nous avons alors $\ell_\g(\widehat Y) \leq \l \ell_\g(\widehat X)$ pour toute classe d'isotopie $\g$ de courbe fermée simple. Remarquez que $\widehat Y$ se rétracte par déformation sur $\widehat X$, il y a donc une identification canonique entre les classes d'isotopie de courbes fermées de $\widehat X$ et celles de $\widehat Y$. \par
 Soit $X$ une surface hyperbolique à bord géodésique non simplement connexe. Le revêtement universel $\tilde X$ se réalise comme un sous-ensemble convexe de $\Hyp$ bordé par une infinité de géodésiques disjointes. Sur $\tilde X$ agit un sous-groupe discret sans torsion $\G\leqslant \Isom(\Hyp)$ tel que $X$ soit isométrique à $\tilde X/\G$.
 % Il ne sert à rien d'ajouter une couronne à un des bord de $X$, puisque le lemme de Schwarz-Pick ne s'applique qu'aux fonctions définies sur le disque tout entier. En effet, la preuve du lemme utilise la transitivité de l'action de $\Isom^+(\D)$ sur $\D$. Pour remédier à ce problème, nous considérons l'extension de Nielsen de $X$ définie par $\widehat X=\Hyp/\G$.
 Nous appelons \emph{extension de Nielsen} de $X$ la surface $\widehat X =\Hyp/\G$. Il s'agit, à isométrie près, de l'unique surface hyperbolique sans bord dont le c\oe ur convexe est isométrique à $X$. Le raisonnement ci-dessus appliqué à l'extension de Nielsen de $X$ donne~:
  
\begin{lemmedecontraction}
Soit $X$ une surface hyperbolique d'aire finie et à bord non vide. Soit $\widehat Y$ une surface hyperbolique obtenue en ajoutant une couronne à l'une des vasques de l'extension de Nielsen $\widehat X$ de $X$. Alors l'inclusion $\widehat X\hookrightarrow \widehat Y$ est contractante, et il existe un nombre réel $0<\l<1$ tel que  $$\ell_\g(\widehat Y)\leq \l \ell_\g(\widehat X)$$
pour toute classe d'isotopie de géodésique fermée simple $\g$. L'inégalité ci-dessus reste évidemment vraie si l'on remplace $\widehat X$ et $\widehat Y$ par leurs c\oe urs convexes $X$ et $Y$.
\end{lemmedecontraction}

% Le lemme ci-dessus fournit un procédé pour contracter les longueurs de toutes les géodésiques fermées simples d'un même facteur $\l<1$, quitte à augmenter de manière arbitrairement petite les longueur des composantes de bord.

\begin{remarks}\label{rem:contraction}\begin{enumerate}
\item Lorsque $X$ est compacte, nous avons $\ell_\g(Y)\leq\l \ell_\g(X)$ avec $0<\l<1$ pour toute géodésiques fermée $\g$.
\item Par densité des multi-géodésiques à poids rationnels, et par continuité de la fonctionnelle longueur, l'inégalité s'étend aux laminations géodésiques mesurées ne contenant aucun arc (voir \textsection~\ref{sec:rappels}).
\item Nous ne savons pas si l'inclusion $\widehat X\hookrightarrow \widehat Y$ envoie $X$ à l'intérieur de $Y$.
\end{enumerate}
\end{remarks}
 
 Ce résultat améliore et simplifie des résultats antérieurs. Ainsi, W.P.~Thurston a montré comment modifier une structure hyperbolique sur une surface à bord de manière à réduire les longueurs de toutes les géodésiques fermées. Son travail n'a pas été publié, mais on trouve un exposé de sa méthode dans l'article \cite{papadopoulos} de A.~Papadopoulos et G.~Théret. Le même résultat a été démontré indépendamment par H.~Parlier (\cite{parlier}). En fait, leurs constructions sont équivalentes (il s'agit de travailler dans un pantalon), mais l'approche de Thurston produit un meilleur contrôle sur les longueurs, du type $\ell_\g(Y)\leq \ell_\g(X)-\e$ pour un certain $\e>0$ indépendant de $\g$.\par

\subsection{Les identités géométriques vues comme obstruction}%%%%%%%%%%%%%
Considérons une surface hyperbolique \emph{orientable} $X$ d'aire finie et à bord non vide. Les géodésiques fermées simples de $X$ satisfont l'identité de McShane-Mirzakhani (\cite{mirzakhani}) ~:
\begin{eqnarray*}
\sum_{\{\g,\d\}} 2 \ln\left(\frac{e^{b_1/2}+e^{\frac{\g+\d}{2}}}{e^{-b_1/2}+e^{\frac{\g+\d}{2}}} \right)+\sum_{i=1}^k \sum_\eta b_1-\ln\left(\frac{\cosh(\frac{b_i}{2})+\cosh(\frac{b_1+\eta}{2})}{\cosh(\frac{b_i}{2})+\cosh(\frac{b_1-\eta}{2})}\right) &=& \frac{b_1}{2},
\end{eqnarray*} 
où $b_1,\ldots, b_k$ sont les composantes du bord $\partial X$, $\{\g,\d\}$ parcourt les paires de géodésiques fermées simples bordant un pantalon avec $b_1$, et $\eta$ parcourt l'ensemble des géodésiques fermées simples bordant un pantalon avec $b_1$ et $b_i$. Afin de faciliter la lecture de la formule, nous utilisons le nom d'une géodésique pour désigner sa longueur.\par
 Les termes généraux des deux séries sont des fonctions strictement croissantes en les variables $b_1,\ldots,b_k$, et strictement décroissantes en les variables $\g,\d,\eta$. Nous en déduisons l'impossibilité d'augmenter les longueurs de toutes les géodésiques fermées simples à longueurs de bords fixées~: 

\begin{proposition}[Mirzakhani]
Soient $X$ et $Y$ deux métriques hyperboliques d'aire finie sur une surface à bord non vide. Supposons que $\ell_{b_i}(X)=\ell_{b_i}(Y)$ pour toute composante de bord $b_i$. Si $\ell_\g(X)\leq\ell_\g(Y)$ pour toute classe d'isotopie de géodésique fermée simple $\g$, alors $X$ et $Y$ sont isotopes.
\end{proposition}

 Considérons une surface hyperbolique $X$ d'aire finie et à bord non vide, nous ne supposons plus $X$ orientable. Les arcs géodésiques orthogonaux au bord satisfont l'identité de Bridgeman (\cite{bridgeman})~:
\begin{eqnarray*}
\sum_\a  \mathcal{R}\left(\frac{1}{\cosh^2(\a/2)}\right) & = & \frac{\pi}{4} \area(X),
\end{eqnarray*}
où $\a$ parcourt l'ensemble des arcs géodésiques orthogonaux à $\partial X$, et $\mathcal{R}$ désigne le dilogarithme de Rogers. Le terme général de la série est une fonction strictement décroissante de la longueur de $\a$. Nous en déduisons  (formule de Gauss-Bonnet) l'impossibilité d'augmenter les longueurs des arcs orthogonaux à topologie fixée~:

\begin{proposition}[Bridgeman]\label{pro:bridgeman}
Soient $X$ et $Y$ deux métriques hyperboliques d'aire finie sur une surface à bord non vide. Si $\ell_\a(X)\leq\ell_\a(Y)$ pour toute classe d'isotopie d'arc $\a$, alors $X$ et $Y$ sont isotopes.
\end{proposition}

Cette proposition répond à une question de Papadopoulos et Théret (\cite{papadopoulos}).

\begin{question}
Peut-on remplacer \emph{arc} par \emph{arc simple} dans la proposition~?
\end{question}

\subsection{Espaces de Teichmüller et laminations géodésiques mesurées}\label{sec:rappels}%%%%%%%%%%%%%%
Dans ce paragraphe, nous introduisons des objets classiques de topologie et géométrie des surfaces. Ils interviendront de manière essentielle dans les preuves des propositions à venir. Pour plus de détails, nous recommandons les livres \cite{imayoshi,hubbard} sur les espaces de Teichmüller, et les livres \cite{bonahon,penner} sur les laminations géodésiques.

\subsubsection*{Espaces de Teichmüller} Soit $S$ une surface admettant une métrique hyperbolique d'aire finie. L'\emph{espace de Teichmüller} $\teich(S)$ est l'espace des classes d'isotopie de métriques hyperboliques sur $S$. C'est une variété lisse difféomorphe à une boule ouverte. Précisons que nous ne supposons pas fixées les longueurs des composantes de bord. À chaque classe d'isotopie de courbe fermée non périphérique $\g$ on associe sa fonction \emph{longueur de géodésique} $\ell_\g:\teich(S)\rightarrow \R_+^\ast$, qui en un point $[X]$ donne la longueur $\ell_\g(X)$ de l'unique géodésique dans la classe $\g$. Ces fonctions sont lisses.
 
\subsubsection*{Métrique de Weil-Petersson} L'espace de Teichmüller $\teich(S)$ admet une métrique riemannienne à courbure sectionnelle négative dite de \emph{Weil-Petersson}. Bien que non complète, cette métrique possède de bonnes propriétés~: elle est uniquement géodésique, et les hessiens des fonctions longueur de géodésique sont définis positifs (S.~Wolpert \cite{wolpert}).\par
% La métrique de Weil-Petersson est habituellement définie lorsque $S$ est orientable et sans bord. Dans les autres cas, $S$ admet un revêtement double orientable sans bord $S^d$, et l'espace de Teichmüller $\teich(S)$ se plonge en une sous-variété  totalement géodésique de $\teich(S^d)$ (voir \cite{paysage} \textsection~6). Cette construction permet de récupérer une métrique possédant les propriétés énoncées ci-dessus.\par
 Le complété de Weil-Petersson de l'espace de Teichmüller s'appelle l'\emph{espace de Teichmüller augmenté}. Sans entrer dans les détails, disons qu'un point du bord consiste en une surface hyperbolique dont certaines géodésiques fermées simples disjointes sont de longueurs nulles (voir H.~Masur \cite{masur}).\par

\subsubsection*{Systole} La \emph{systole} d'une métrique hyperbolique sur $S$ est la longueur de sa plus courte géodésique fermée non périphérique. En tant que minimum des fonctions longueur de géodésique, la systole définit une fonction continue sur $\teich(S)$. Si $S$ est fermée, alors la systole est bornée supérieurement sur l'espace de Teichmüller.\par
 
\subsubsection*{Laminations géodésiques mesurées} 
Fixons une métrique hyperbolique $X$ sur $S$. Une \emph{lamination géodésique} $\eta$ de $X$ est un fermé de $S$ qui se décompose en une union disjointe de
\begin{itemize}
\item géodésiques fermées simples pouvant être des composantes de bord,
\item géodésiques simples infinies n'allant pas à l'infini dans une région cuspidale,
\item arcs géodésiques simples orthogonaux au bord $\partial S$ en leurs extrémités.
\end{itemize}
 Une \emph{mesure transverse} à $\eta$ est une mesure définie sur chaque arc transverse à $\eta$, et invariante par les isotopies préservant $\eta$. Une \emph{lamination géodésique mesurée} est une lamination géodésique munie d'une mesure transverse dont le support est la lamination toute entière. Les laminations géodésiques mesurées les plus simples sont celles dont toutes les feuilles sont compactes (nécessairement en nombre fini).\par
 Soit $\eta$ une lamination géodésique formée d'un nombre fini de feuilles compactes. Toute mesure transverse à $\eta$ s'écrit comme une combinaison linéaire à c\oe fficients positifs $\sum_{i}  a_i i(\eta_i,\cdot)$, où $\eta_i$ parcourt l'ensemble des feuilles de $\eta$, et $i(\cdot,\cdot)$ désigne la fonction \emph{nombre d'intersection}. Nous notons $\sum_i a_i \eta_i$ la lamination géodésique mesurée associée, et nous définissons sa longueur par $\ell_{\sum_i a_i \eta_i}(X)=\sum_i a_i \ell(\eta_i)$.\par
 L'ensemble $\ML(X)$ des laminations géodésiques mesurées de $X$ possède une structure linéaire par morceaux \emph{entière}, pour laquelle les points entiers sont les combinaisons linéaires à c\oe fficients \emph{entiers} $\sum_i a_i \eta_i$. La multiplication des mesures transverses par les nombres réels positifs fait de $\ML(X)$ un cône. Les points rationnels de $\ML(X)$ sont les multiples rationnels des points entiers.\par
 Soit $\simple(X)$ l'ensemble des géodésiques fermées simples ne bordant ni un disque ni un ruban de M\oe bius (parmi lesquelles nous comptons les composantes de bord). L'espace $\ML(X)$ se plonge topologiquement dans l'espace affine $\R^{\simple(X)}$ muni de la topologie produit. Le projectifié $\PML(X)$ est homéomorphe à une sphère de dimension finie. La fonctionnelle longueur $\eta\mapsto \ell_\eta(X)$ définie plus haut se prolonge en une fonction continue sur $\ML(X)$.\par
 \'Etant donnée une deuxième métrique hyperbolique $Y$, les espaces $\ML(X)$ et $\ML(Y)$ sont canoniquement isomorphes relativement à leurs structures linéaires par morceaux entières. Aussi, nous parlerons de laminations géodésiques mesurées sans avoir fixée au préalable une métrique hyperbolique sur $S$, et nous utiliserons les notations transparentes $\ML(S)$ et $\PML(S)$.

\subsection{Contraction de tous les arcs sauf d'un nombre fini}%%%%%%%%%%%%%%
Commençons par rappeler un fait bien connu~:
\begin{proposition}
Soient $X$ et $Y$ deux métriques hyperboliques d'aire finie sur une surface à bord non vide $S$. S'il existe une géodésique orientable $\g$ telle que $\ell_\g(Y)<\ell_\g(X)$, alors l'inégalité $\ell_\a(Y)<\ell_\a(X)$ est satisfaite par une infinité d'arcs simples $\a$.
\end{proposition}

\begin{proof}
Soit $U\subset \ML(S)$ le cône ouvert formé des laminations géodésiques mesurées $\eta$ telles que $\ell_\eta(Y)<\ell_\eta(X)$. Cet ouvert se projette sur un ouvert $\bar U$ de $\PML(S)$ contenant $\g$.\par
 Soit $\tau_\g$  le twist de Dehn selon la géodésique $\g$. \'Etant donné un arc simple $\a$ intersectant $\g$, la suite $(\tau_\g^n(\a))_n$ converge vers $\g$ dans $\PML(S)$ (proposition~3.4 de \cite{farb}). En conséquence, $\bar U$  contient une infinité d'arcs $\tau_\g^n(\a)$.
\end{proof}

Les propositions suivantes montrent que l'on peut décroître les longueurs de tous les arcs sauf d'un nombre fini. Selon la proposition~\ref{pro:bridgeman} on ne peut pas faire mieux.

\begin{proposition}
Soient $X$ et $Y$ deux métriques hyperboliques d'aire finie sur une surface à bord non vide $S$. S'il existe un nombre réel $0<\l<1$ tel que $\ell_\g(Y)\leq\l \ell_\g(X)$ pour toute classe d'isotopie de géodésique fermée simple $\g$. Alors il y a seulement un nombre fini de classes d'isotopie d'arcs \emph{simples} $\a$ telles que $\ell_\a(Y)\geq\ell_\a(X)$.
\end{proposition}

\begin{remark}\begin{enumerate}
\item La proposition est probablement vraie pour tous les arcs, et pas seulement pour les arcs simples. La proposition ci-dessous va dans ce sens.
\item La même preuve permet de montrer que, pour tout $\e>0$, il y a seulement un nombre fini d'arcs simples $\a$ tels que $\ell_\a(Y)\geq(\l+\e) \ell_\a(X)$.
\end{enumerate}
\end{remark}

\begin{proof}
Soit $U\subset\ML(S)$ le cône ouvert formé des laminations géodésiques mesurées $\eta$ telles que $\ell_\eta(Y)<\ell_\eta(X)$. Cet ouvert se projette sur un ouvert $\bar U$ de $\PML(S)$. Considérons $(\a_n)_n$ une suite d'arcs simples \emph{distincts} convergeant vers un point $\a_\infty$ dans $\PML(S)$. Pour prouver la proposition, il suffit de montrer que $\a_\infty$ appartient à $\bar U$.\par
 Les $\a_n$ sont des points entiers distincts de $\ML(S)$. De ce fait, la suite $(\a_n)_n$ n'est contenue dans aucun sous-ensemble compact de $\ML(S)$. Nous en déduisons qu'il existe une géodésique fermée simple $\g_0$ telle que la suite $(i(\a_n,\g_0))_n$ ne soit pas bornée. Quitte à prendre une sous-suite, nous supposons que $i(\a_n,\g_0)$ tend vers l'infini avec $n$.\par
 Pour toute composante de bord $b_i$, le rapport $i(\a_n,b_i)/i(\a_n,\g_0)\leq 2/i(\a_n,\g_0)$ tend vers $0$ quand $n$ tend vers l'infini. Nous en déduisons qu'aucune feuille de la lamination sous-jacente à $\a_\infty$ n'est un arc. Suivant la remarque~\ref{rem:contraction} (2), la classe projective $\a_\infty$ appartient à $\bar U$.
\end{proof}

\begin{proposition}
Soit $X$ une surface hyperbolique compacte à bord non vide. Soit $Y$ le c\oe ur convexe d'une surface hyperbolique $\widehat Y$ obtenue en ajoutant une couronne à l'une des vasques de l'extension de Nielsen $\widehat X$ de $X$. Alors, seul un nombre fini d'arcs $\a$ satisfont $\ell_\a(Y)\geq\ell_\a(X)$. 
\end{proposition}

\begin{proof}
Pour chaque composante de bord $b_i$ de $X$, nous choisissons une isotopie $H_i$ entre $b_i$ et la géodésique de $\widehat Y$ isotope à $b_i$. Nous notons $A$ la longueur maximale des chemins $t\mapsto H_i(t,x)$ où $x$ est un point du bord $\partial X$.  
Par le lemme de contraction, il existe $0<\l<1$ tel que l'inclusion $\widehat X\hookrightarrow \widehat Y$ contracte les longueurs d'un facteur $\l$.\par
 Soit $a:[0,1]\rightarrow X$ un arc géodésique orthogonal à $\partial X$ en ses extrémités, et soit $\a$ la classe d'isotopie de cet arc. En concaténant $a$ avec les chemins $t\mapsto H_{i_0}(t,a(0))$ et $t\mapsto H_{i_1}(t,a(1))$ nous obtenons un arc de $\widehat Y$ de longueur majorée par $\l\ell(a)+2A$. Nous en déduisons que si $\ell_\a(X)\geq 2A/(1-\l)$ alors $\ell_\a(Y)\leq\ell_\a(X)$.\par
  Nous concluons en remarquant qu'il y a seulement un nombre fini de classes d'isotopie d'arcs $\a$ de longueurs $\ell_\a(X)<2A/(1-\l)$.
\end{proof}

\subsection{Non contraction des géodésiques}%%%%%%%%%%%%%%%
 Nous présentons une preuve courte et originale d'un résultat bien connu~: 

\begin{proposition}[Thurston]
Soient $X$ et $Y$ deux métriques hyperboliques sur une surface fermée $S$. Si $\ell_\g(Y)\leq\ell_\g(X)$ pour toute classe d'isotopie de courbe fermée simple $\g$, alors les métriques $X$ et $Y$ sont isotopes.
\end{proposition}

\begin{proof}
Par l'absurde, nous supposons que $X$ et $Y$ ne sont pas isotopes, et que $\ell_\g(Y)\leq\ell_\g(X)$ pour toute classe d'isotopie $\g$. Les points $[X]$ et $[Y]$ de l'espace de Teichmüller $\teich(S)$ étant distincts, il existe une géodésique de Weil-Petersson $t\mapsto c(t)$ telle que $c(0)=[Y]$ et $c(t_0)=[X]$ pour un certain $t_0>0$.\par
 Pour toute classe d'isotopie $\g$, l'hypothèse $\ell_\g(Y)\leq\ell_\g(X)$ et la stricte convexité de $\ell_\g\circ c$ entraînent que la dérivée de $\ell_\g\circ c$ est minorée par une constante positive pour les temps $t\geq t_0$. Nous en déduisons que la systole croît strictement le long de la géodésique $c$ pour les temps $t\geq t_0$. En conséquence, la géodésique $c$ est complète en les temps positifs (théorème de Masur, voir \textsection~\ref{sec:rappels}).\par
 Fixons un nombre réel positif $A$, et notons $\g_1,\ldots,\g_n$ les géodésiques fermées de $X$ de longueur au plus $A$. Nous venons de voir que les dérivées des fonctions $\ell_{\g_i}\circ c$ sont minorées par une même constante positive sur $[t_0,+\infty[$. Comme les autres fonctions longueur de géodésique sont croissantes le long de $c$, nous en déduisons que $\sys(c(t))\geq A$ pour les temps $t$ suffisamment grand. Ceci contredit le fait que la systole est bornée sur l'espace de Teichmüller.
\end{proof}

%\begin{question}
%Peut-on déduire ce résultat d'une identité géométrique~?
%\end{question}

\begin{corollary}
Soient $X$ et $Y$ deux métriques hyperboliques non isotopes sur une surface fermée orientable $S$. Alors il existe une infinité de classes d'isotopie de courbes fermées simples $\g$ telles que $\ell_{\g}(Y)<\ell_{\g}(X)$.
\end{corollary}

\begin{proof}
Soit $U\subset\ML(S)$ le cône ouvert formé des laminations géodésiques mesurées $\eta$ telles que $\ell_\eta(Y)<\ell_\eta(X)$. Cet ouvert se projette sur un ouvert $\bar U$ de $\PML(S)$. Selon la proposition précédente, les ouverts $U$ et $\bar U$ sont non vides.\par
 Comme l'action du groupe modulaire sur $\PML(S)$ est minimale, il vient que $\bar U$ contient une infinité de courbes fermées simples.
\end{proof}

%%%%%%%%%%%%%%%%%%%%%%%%%%%%%%
\section{Longueur minimale des lacets avec auto-intersection}\label{sec:intersection}
%%%%%%%%%%%%%%%%%%%%%%%%%%%%%%

 Dans cette partie, nous étendons un théorème de A.~Yamada (\cite{yamada}) aux surfaces non orientables.
Pour des références et une preuve différente de celle de Yamada on peut consulter le \textsection~4.2 de \cite{buser}. 

\begin{theorem}
Si $\g$ est une géodésique fermée primitive et non simple d'une surface hyperbolique (éventuellement à bord géodésique), alors
$\ell(\g)\geq 2\mrm{arccosh}(3)$.
Le cas d'égalité est uniquement réalisé par les trois géodésiques réalisant la systole du pantalon à trois pointes.
\end{theorem}

\begin{proof}
Nous nous ramenons au cas où $\g$ est homéomorphe à un huit en suivant la preuve du théorème~4.2.4 de \cite{buser}. Nous appelons $X$ la plus petite sous-surface à bord géodésique contenant $\g$. Cette sous-surface se rétracte sur un voisinage tubulaire de $\g$, d'où sa caractéristique d'Euler-Poincaré vaut $-1$. Ainsi $X$  est un pantalon, une bouteille de Klein à un bord, ou un plan projectif à deux bords. Ces différents cas sont représentés sur la première colonne de la figure~\ref{fig:trigonometrie}, les bords grisés sont auto-recollés.\par
  
\begin{figure}[ht]
\centering
\psfrag{A}{$\g(t_0)$}\psfrag{B}{$\g(t_2)$}
\includegraphics[totalheight=9cm,keepaspectratio=true]{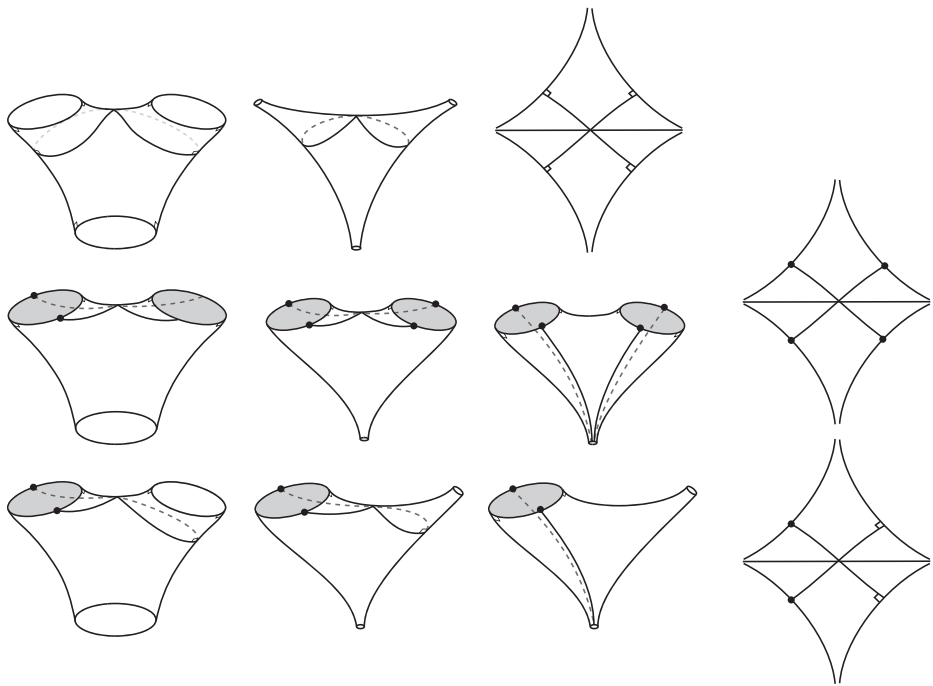}
\caption{Configurations topologiques}\label{fig:trigonometrie}
\end{figure}
  
 Quitte à appliquer le lemme de contraction (\textsection~\ref{sec:geodesics}), nous supposons les bords de $X$ de longueur nulle (deuxième colonne de la figure~\ref{fig:trigonometrie}). En découpant trois arcs, nous décomposons $X$ en deux triangles idéaux. Un, deux ou tous les arcs sont supportés par le lieu des points fixes de l'involution hyperelliptique $\iota_X$ de $X$. Rappelons que, dans chaque pantalon de la figure~\ref{fig:trigonometrie}, ce lieu des points fixes consiste en l'union des perpendiculaires communes (voir la partie~1 de \cite{paysage}). Nous complétons les arcs fixés point-à-point par $\iota_X$ en un système maximal d'arcs comme indiqué sur la troisième colonne de la figure~\ref{fig:trigonometrie}. Le système d'arcs obtenu est globalement invariant par $\iota_X$, et au moins un des arc est fixé point-à-point par $\iota_X$. Ainsi, les deux triangles idéaux sont images l'un de l'autre par la réflexion suivant un des arcs (quatrième colonne de la figure~\ref{fig:trigonometrie}). Nous voyons tout de suite que la longueur de $\g$ est supérieure à la longueur des perpendiculaires communes représentées sur le dessin de droite de la première ligne de la figure~\ref{fig:trigonometrie}. 
\end{proof}

%%%%%%%%%%%%%%%%%%%%%%%%%%%%%%%%%%%%%
\section{Minoration du rayon d'injectivité des surfaces hyperboliques}\label{sec:yamada}
%%%%%%%%%%%%%%%%%%%%%%%%%%%%%%%%%%%%%

Dans cette partie, nous donnons une preuve nouvelle du théorème ci-dessous dû à A.~Yamada (\cite{yamada}). Cette preuve a l'avantage d'être particulièrement courte, complète (nous traitons aussi le cas d'égalité), et ne nécessite que peu de calculs. Une autre preuve du théorème de Yamada a aussi été obtenue par F.~Fanoni en utilisant les méthodes développées dans \cite{fanoni}.\par

\begin{notations}
\'Etant donnée une surface hyperbolique $X$, nous notons $\ri_x(X)$ le rayon d'injectivité au point $x\in X$, et $\ri(X)=\sup_{x\in X} \ri_x(X)$ la borne supérieure du rayon d'injectivité sur $X$.
\end{notations}

\begin{theoremnonumber}[Yamada]\label{th:yamada}
Soit $X$ une surface hyperbolique orientable sans bord. Nous avons $\ri(X)\geq\mrm{arcsinh}(2/\sqrt{3})$, avec égalité si et seulement si $X$ est isométrique à la sphère à trois pointes.
\end{theoremnonumber}

\begin{remark}\begin{enumerate}
\item La surface n'est pas supposée d'aire finie.
\item Fanoni (\cite{fanoni}) a établi une inégalité du même type pour les orbisurfaces hyperboliques.
\end{enumerate}
\end{remark}

Notre preuve tient en deux étapes~: on se ramène au cas à pointes (lemme~\ref{lem:minoration} et proposition~\ref{pro:reduction}), puis on traite ce cas en utilisant la géométrie des régions cuspidales (proposition~\ref{pro:pointes}). Nous ferons appel à différentes techniques~: ligne de partage, lemme de Schwarz, cercles isométriques, empilement d'horodisques. Nous effectuerons au \textsection~\ref{sec:cercles}  des rappels concernant les régions cuspidales et les cercles isométriques. Les configurations de cercles isométriques expliquent (à travers la proposition~\ref{pro:sorvali} et le lemme~\ref{lem:minoration-distance}) pourquoi la borne $\mrm{arcsinh}(2/\sqrt{3})$ n'est plus valable en dimensions supérieures (comparer avec le théorème~1.3 de \cite{gendulphe}).\par

\subsection{Existence d'une petite géodésique fermée dans le cas sans pointe}%%%%%%%%%%%%

\begin{lemma}\label{lem:minoration}
Soit $X$ une surface hyperbolique orientable sans pointe n'admettant aucune géodésique fermée de longueur inférieure à $2\mrm{arcsinh}(\sqrt{3}/2)$. Alors le rayon d'injectivité de $X$ vérifie $\ri(X)>\mrm{arcsinh}(2/\sqrt{3})$.
\end{lemma}

\begin{proof}
Nous supposons que $X$ satisfait les hypothèses de l'énoncé. Soit $(\delta_j)_{j\in J}$ une famille maximale de géodésiques fermées simples disjointes satisfaisant 
$$2\mrm{arcsinh}(2/\sqrt{3})\geq \ell(\delta_j)\geq 2\mrm{arcsinh}(\sqrt{3}/2).$$
Les voisinages collier de largeur $\mrm{arcsinh}(\sqrt{3}/2)$ autour des $\delta_j$ sont disjoints (voir \cite{buser} \textsection~4.1). Nous supposons la famille $(\delta_j)_J$ non vide, et la fonction distance $x\mapsto d_X(x,\cup_J \delta_j)$ majorée. Si l'une de ces hypothèses est contredite, alors le lemme est vérifié.\par
 Pour commencer, nous montrons qu'il existe un point $x_0\in X$ à une distance supérieure à $\mrm{arcsinh}(2/\sqrt{3})$ de $\cup_J\delta_j$. Nous notons $(\tilde\delta_l)_{l\in L}$ la famille des géodésiques du revêtement universel $\tilde X$ au-dessus des $\delta_j$. L'ensemble des points de $\tilde X$ dont la distance à $\cup_L\tilde\delta_l$ est réalisée par au moins deux $\tilde\delta_l$ s'appelle la \emph{ligne de partage} (voir \cite{bavard05}). Il s'agit d'un graphe géodésique, dont la projection partage $X$ en anneaux. Comme  $x\mapsto d_X(x,\cup_J \delta_j)$ est supposée bornée, la ligne de partage admet au moins un sommet $\tilde x_0$. La distance $d$ entre $\tilde x_0$ et $\cup_L \tilde\delta_l$ est réalisée par au moins trois géodésiques $\tilde\delta_1,\tilde\delta_2,\tilde\delta_3$. Quitte à changer les indices, nous supposons que l'angle $\f$ entre les segments issus de $\tilde x_0$ orthogonaux à $\tilde\delta_1$ et $\tilde\delta_2$ est au plus $2\pi/3$. Dans le pentagone de la figure~\ref{fig:pentagone}, le côté opposé à $\tilde x_0$ est de longueur supérieure à $2 \mrm{arcsinh}(\sqrt{3}/2)$ (en raison des voisinages collier). Par trigonométrie dans un trirectangle moitié du pentagone nous trouvons $d>\mrm{arcsinh}(2/\sqrt{3})$. Nous notons $x_0$ la projection de $\tilde x_0$ dans $X$.\par
  
\begin{figure}[ht]
\centering
\psfrag{x}{$\tilde x_0$}\psfrag{F}{$\tilde\delta_1$}\psfrag{G}{$\tilde\delta_2$}\psfrag{d}{$d$}\psfrag{p}{$\f$}\includegraphics[totalheight=3cm,keepaspectratio=true]{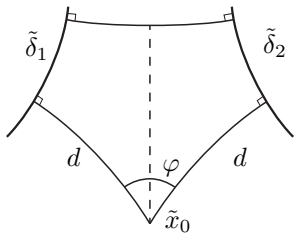}
\caption{Le pentagone}\label{fig:pentagone}
\end{figure}  
  
 Maintenant, nous minorons la longueur des lacets basés en $x_0$ en distinguant trois cas. Soit $\g$ un lacet géodésique basé en $x_0$. Si $\g$ est isotope à une géodésique disjointe de $\cup_J\delta_j$, alors $\ell(\g)>2\mrm{arcsinh}(2/\sqrt{3})$ par définition de $(\delta_j)_{J}$. Si $\g$ est isotope à une géodésique $\delta_j$, alors $\sinh(\g/2)\geq \sinh(\delta_j/2)\cosh(d)> \sqrt{7}/2$. Si $\g$ intersecte $\cup_J\delta_j$, alors $\ell(\g)\geq 2d_X(x_0,\cup_J\delta_j)>2\mrm{arcsinh}(2/\sqrt{3})$.
%  Le lemme est prouvé, sauf s'il y a égalité dans la dernière inégalité ci-dessus. Dans ce cas, les segments géodésiques réalisant la distance entre $\tilde x_0$ et $\cup_L\tilde\delta_l$ sont exactement au nombre de trois, et les angles formés par ces segments sont tous égaux à $2\pi/3$. De plus, si $\tilde x_i$ désigne la projection orthogonale $\tilde x_0$ sur $\tilde\delta_i$, alors deux points parmi les $\tilde x_i$ se projettent sur un même point de $X$. Pour simplifier nous supposons qu'il s'agit de $\tilde x_1$ et $\tilde x_2$, et nous notons $x_1$ leur projection sur $X$. Le point $\tilde x_3$ ne peut se projeter sur $x_1$, car il y a au plus deux segments issus de $x_0$ orthogonaux à une même géodésique $\delta_j$ en un point donné $x_1$. En nous éloignant un peu de $\tilde x_0$ dans la direction de $\tilde \delta_3$, nous trouvons un point dont le projeté sur $X$ a un rayon d'injectivité supérieur à $\mrm{arcsinh}(2/\sqrt{3})$.
 \end{proof}

\subsection{Réduction au cas à pointes}%%%%%%%%%%%%%%%%%

\begin{proposition}\label{pro:reduction}
Pour toute surface hyperbolique orientable sans pointe $X$, il existe une surface hyperbolique orientable à pointes $Y$ telle que $\ri(X)\geq\ri(Y)$.\end{proposition}

\begin{remark}
Si $\ri(X)\leq \mrm{arcsinh}(2/\sqrt{3})$ alors nous pouvons prendre $Y$ avec au plus deux pointes (voir la preuve ci-dessous), donc distincte du pantalon à trois pointes.
\end{remark}

\begin{proof}
Nous supposons $\ri(X)\leq\mrm{arcsinh}(2/\sqrt{3})$, sinon nous prenons $Y$ isométrique au pantalon à trois pointes. Selon le lemme~\ref{lem:minoration}, il existe une géodésique fermée simple $\delta$ de longueur inférieure à $2\mrm{arcsinh}(\sqrt{3}/2)$.\par
 Soit $X'$ l'extension de Nielsen d'une composante connexe de $X\setminus\delta$. La surface $X'$ contient une ou deux géodésiques s'identifiant à $\delta$. Chaque bout de $X'$ se rétractant sur l'une de ces géodésiques admet un bord conforme, %dans le sens où l'angle entre deux courbes issues d'un point du bord est bien défini. Il suffit de travailler dans le revêtement universel identifié au disque unité pour s'en convaincre.
  Nous collons un disque épointé sur chacun de ces bords, et nous obtenons ainsi une surface $Y$ munie d'une structure conforme, et par suite d'une métrique hyperbolique. D'après le lemme de Schwarz (voir \textsection~\ref{sec:geodesics}), l'inclusion $X' \hookrightarrow Y$ réduit les distances.\par
 Le rayon d'injectivité (relativement à la surface $X$) en un point quelconque $x\in X\setminus\delta$ est réalisé par un lacet géodésique $\g$ de longueur inférieure ou égale à $2\mrm{arcsinh}(2/\sqrt{3})$. Ce lacet ne peut intersecter $\delta$ sans contredire le lemme du collier (voir \cite{buser}), il est donc contenu dans $X\setminus \delta$. Comme l'inclusion réduit les distances nous avons $\ell_X(\g)>\ell_Y(\g)$, et le rayon d'injectivité de $Y$ en $x$ est inférieur au rayon d'injectivité de $X$ en $x$. Nous contrôlons facilement le rayon d'injectivité dans $Y\setminus X'$. Finalement nous avons bien $\ri(X)\geq\ri(Y)$.
\end{proof}

\subsection{Minoration de $\ri$ pour les surfaces à pointes}

\begin{proposition}\label{pro:pointes}
Une surface hyperbolique orientable à pointes vérifie $$\ri\geq\mrm{arcsinh}(2/\sqrt{3}),$$
avec égalité si et seulement si la surface est isométrique au pantalon à trois pointes.
\end{proposition}

\begin{proof}
Soit $X$ une surface hyperbolique orientable à pointes, que nous supposons non isométrique à la sphère à trois pointes.
Nous allons montrer que le rayon d'injectivité en un certain point $x_0$ est supérieur à $\mrm{arcsinh}(2/\sqrt{3})$.\par 
 Nous choisirons $x_0$ à l'intérieur d'une région cuspidale maximale $C$. Remarquez que l'aire de $C$ est supérieure à $4$ par un théorème de Seppälä et Sorvali (nous renvoyons à l'appendice pour plus de détails, proposition~\ref{pro:sorvali}). Dans la suite, nous notons $h_x$ l'horocycle passant par un point $x\in C$, et $c_x$ le lacet géodésique d'origine $x$ homotope à $h_x$.\par
\emph{1) Le point $x_0$.}  Considérons l'ensemble $C_0$ formé
des points $x\in C$ en lesquels le rayon d'injectivité est réalisé par $c_x$. Nous supposons $C_0$ contenu dans chacune des sous-régions cuspidales d'aire $4/\sqrt{3}+\e$ avec $\e>0$. Dans le cas contraire, la proposition serait trivialement vérifiée vue la relation $\sinh(\ell(c_x)/2) =  \ell(h_x)/2$. Nous en déduisons l'existence d'un point $x_0\in C_0$ réalisant la borne supérieure du rayon d'injectivité sur $C_0$ (travailler sur un sous-ensemble compact de $C_0$).\par
 \emph{2) Il y a au moins trois lacets géodésiques réalisant le rayon d'injectivité en $x_0$.} Sinon, en se déplaçant le long de l'horocycle $h_{x_0}$, nous trouverions un point $x_1$ en lequel le rayon d'injectivité serait uniquement réalisé par le lacet $c_{x_1}$. En partant de $x_1$ et en suivant la géodésique orthogonale à l'horocycle $h_{x_0}$, nous augmenterions le rayon d'injectivité tout en restant dans $C_0$. Ceci contredirait la définition de $x_0$.\par 
 Introduisons quelques notations~: nous appelons $S\subset\pi_1(X,x_0)$ l'ensemble des classes d'homotopie des lacets réalisant le rayon d'injectivité en $x_0$, nous notons $(\g_i)_{i\in I}$ la famille des géodésiques fermées simples appartenant aux classes d'homotopie libre induites par $S$.\par
\emph{3) Les géodésiques $\g_i$ sont disjointes.}
 Nous supposons que ces géodésiques ne pénètrent pas dans $C$, sinon la proposition serait vérifiée vu le lemme~\ref{lem:minoration-distance}. Nous en déduisons $d_X(x_0,\g_i)\geq d_X(x_0,\partial C)> \mrm{arccosh}(2/\sqrt{3})$, car $C_0$ est contenu dans l'adhérence de la région cuspidale d'aire $4/\sqrt{3}$. Nous supposons que les $\g_i$ sont de longueur inférieure à $2\mrm{arcsinh(1)}$, sinon la proposition serait vérifiée vue la formule~\eqref{eq:deplacement}. Ceci implique que les $\g_i$ sont disjointes par le lemme du collier.\par
 
 \emph{4) Application du lemme de Schwarz afin de minorer le rayon d'injectivité de $X$ en $x_0$ par la longueur de lacets périphériques $\a_s$ d'une surface $Y$.} Nous construisons $Y$ comme dans la preuve de la proposition~\ref{pro:reduction}. Nous prenons l'extension de Nielsen $X'$ de la composante connexe de $X\setminus(\cup_I \g_i)$ contenant $C$, puis nous collons des disques épointés sur les bouts provenant des $\g_i$. De cette façon, nous obtenons une surface hyperbolique $Y$ telle que l'inclusion $X'\hookrightarrow Y$ soit conforme, et par conséquent réduise les distances (lemme de Schwarz). Soit $\a_s$ le lacet géodésique de $Y$ dans la classe $s\in S$, nous avons $2\ri_{x_0}(X)> \ell_Y(\a_s)$.\par
 
 \emph{5) Minoration de $ \sup_{S} \ell(\a_s)$.} Dans le revêtement universel $\tilde Y$, nous fixons un relevé $\tilde x_0$ de $x_0$. À chaque classe $s\in S$ correspond un automorphisme parabolique $\d_s$ de $\tilde Y$ tel que le segment $\tilde x_0\delta_s(\tilde x_0)$ se projette sur $\a_s$.  Soit $B_s$ l'horodisque stable par $\delta_s$ qui se projette sur une région cuspidale d'aire $2$. En notant $d_{\tilde Y}(\tilde x_0,B_s)$ la distance orientée (négative si $\tilde x_0\in B_s$), nous avons $\sinh(\a_s/2)=\exp(d_{\tilde Y}(\tilde x_0,B_s))$. Ainsi, minorer $\sup_S \a_s$ revient à minorer $\sup_S d_{\tilde Y}(\tilde x_0,B_s)$.\par
  Soient $s_1,s_2\in S$ tels que l'angle entre les demi-géodésiques issues de $\tilde x_0$ et orthogonales à $B_{s_1}$ et $B_{s_2}$ soit inférieur ou égal à $2\pi/3$. On montre facilement que la quantité 
 $\sup(d_{\tilde Y}(\tilde x_0,B_{s_1}),d_{\tilde Y}(\tilde x_0,B_{s_2}))$
  est minimale lorsque les horodisques $B_{s_1}$ et $B_{s_2}$ sont tangents et $d_{\tilde Y}(\tilde x_0,B_{s_1})=d_{\tilde Y}(\tilde x_0,B_{s_2})$ (rappelons que les $B_s$ sont disjoints, voir remarque~\ref{rem:disjoints}). Ainsi $\sup_S\sinh(\ell(\a_s)/2)\geq2/\sqrt{3}$ avec égalité si et seulement si $Y$ est isométrique à la sphère à trois pointes.
\end{proof}

%%%%%%%%%%%%%%%%%%%%%%%%%%%%%%%%%%%%%%%%%%%%
%   bibliographie
%%%%%%%%%%%%%%%%%%%%%%%%%%%%%%%%%%%%%%%%%%%%

%\nocite{*}
\bibliographystyle{alpha}
\bibliography{biblio}

%%%%%%%%%%%%%%%%%%%%%%%%%%%%%%%%%%%%%%%%%%%%
%   appendices
%%%%%%%%%%%%%%%%%%%%%%%%%%%%%%%%%%%%%%%%%%%%

\newpage
\appendix

\section{Cercles isométriques et bouts cuspidaux}\label{sec:cercles-isom}%%%%%%%%%%%%%%%%

 Nous revenons sur le résultat de M.~Seppälä et T.~Sorvali (\cite{sorvali}) affirmant que toute pointe d'une surface hyperbolique admet un voisinage consistant en une région cuspidale d'aire $4$. Leur méthode utilise les cercles isométriques, nous reprenons cette idée pour prouver le lemme~\ref{lem:minoration-distance} intervenant dans la preuve de la proposition~\ref{pro:pointes}. Les résultats de ce paragraphe sont déjà connus, à l'exception du lemme~\ref{lem:minoration-distance}.\par
  Rappelons qu'une \emph{région cuspidale} est une partie d'une surface hyperbolique isométrique au quotient de l'horodisque $B_\infty=\{z\in\Hyp~;~\Imaginary(z)>1\}$ par un groupe engendré par une transformation $z\mapsto z+\omega$ avec $\omega>0$.\par

\subsubsection{Rappels sur les cercles isométriques}\label{sec:cercles}
 Considérons une homographie
 $$ f(z)=\frac{az+b}{cz+d}\ \textnormal{avec}\ a,b,c,d\in\R\ \textnormal{tels que}\ ad-bc=1.$$
Nous supposons que $f$ ne fixe pas l'infini. L'ensemble $I(f)$ des points $z$ tels que la différentielle $\mrm{d}f(z)$ soit une isométrie est appellé le \emph{cercle isométrique} de $f$. Nous pouvons le caractériser de différentes manières, le lemme suivant est bien connu~: 
\begin{lemma}
Nous avons~:
\begin{enumerate}[i)]
\item $I(f)=\{z\in\C~;~|cz+d|=1\}$,
\item $I(f)=\{z\in\C~;~ f(z)-z\in \R\}$.
\end{enumerate}
\end{lemma}

Regardons plus en détail le cas où $f$ est \emph{hyperbolique}. Pour simplifier, nous supposons que l'axe de $f$ est le cercle centré en l'origine de rayon $r$. Nous notons $d$ la distance de translation de $f$, et $\theta$ l'angle entre les points de l'axe de $f$ se situant à une distance $d/2$ du point $ri$. Clairement, les caractérisations \emph{i)} et \emph{ii)} impliquent~:
\begin{lemma}
Les cercles isométriques $I(f)$ et $I(f^{-1})=f(I(f))$ ont leurs centres sur l'axe réel et sont orthogonaux à l'axe de $f$ en les points $re^{i(\frac{\pi}{2}\pm\theta)}$. En particulier $I(f)$ et $I(f^{-1})$ sont disjoints.
\end{lemma} 
Des calculs simples donnent les relations
\begin{eqnarray*}
 \sinh(d/2) & = & \tan\theta,\\
 r & = & R \tan\theta .
\end{eqnarray*}
où $R$ désigne le rayon des cercles isométriques $I(f)$ et $I(f^{-1})$.

\subsubsection{Bouts cuspidaux} Les cercles isométriques s'avèrent utiles pour regarder comment une homographie $f$ déplace un horodisque $B_\infty\subset\Hyp$ centré à l'infini. Comme $f$ envoie l'extérieur de $I(f)$ sur l'intérieur de $I(f^{-1})$, il vient~:
\begin{lemma}\label{lem:tangence}
Les horodisques $B_\infty$ et $f(B_\infty)$ sont disjoints (resp. tangents) si et seulement si ils sont disjoints de $I(f^{-1})$ (resp. tangents à $I(f^{-1})$).
\end{lemma}

Considérons un groupe fuchsien non élémentaire $\Gamma\leqslant\PSL(2,\R)$ possédant des éléments paraboliques. Nous appelons $X$ la surface hyperbolique quotient $\Hyp/\G$. Quitte à conjuguer $\G$ dans $\PSL(2,\R)$, nous supposons le stabilisateur de l'infini de la forme $\G_\infty=\langle z\mapsto z+\omega \rangle$ avec $\omega>0$. Quitte à conjuguer une nouvelle fois, nous supposons que $B_\infty=\{z\in\Hyp~;~\Imaginary(z)>1\}$ est le plus grand horodisque centré en l'infini tel que $\g(B_\infty)\cap B_\infty=\emptyset$ pour tout $\g\in\G\setminus\G_\infty$. Par maximalité, l'un des horodisques $\g(B_\infty)$ est tangent à $B_\infty$.

\begin{proposition}[Seppälä-Sorvali]\label{pro:sorvali}
Nous avons $\omega\geq 4$, ainsi l'horodisque $B_\infty$ se projette dans $X$ sur une région cuspidale d'aire au moins $4$. De plus, il y a égalité si et seulement si $X$ est isométrique au pantalon à trois pointes.
\end{proposition}

\begin{remark}\label{rem:disjoints}
\begin{enumerate}
\item La preuve ci-dessous est une version très légèrement modifiée de celle de \cite{sorvali}, plus dans l'esprit de \cite{adams}. Nous déterminons le cas d'égalité, contrairement à \cite{sorvali}.
\item On montre plus facilement (sans recourir aux cercles isométriques) que toute pointe de $X$ admet pour voisinage une région cuspidale d'aire $2$, et que ces voisinages sont disjoints. 
\end{enumerate}
\end{remark}

\begin{proof}
Soit $\g$ un élément de $\G\setminus\G_\infty$ envoyant $B_\infty$ sur un horodisque qui lui est tangent, nous supposons cet horodisque centré en $0$ et l'appelons $B_0$. L'isométrie $\g^{-1}$ envoie les horodisques $B_0$ et $B_\infty$ sur deux horodisques tangents, plus précisément sur $B_\infty=\g^{-1}(B_0)$ et $B_b=\g^{-1}(B_\infty)$. L'indice $b\in \R$ désigne le centre de l'horodisque. Quitte à conjuguer $\g$ par une puissance de $z\mapsto z+\omega$, nous supposons $b$ de module minimal parmi les points de $\G_\infty\cdot b$.\par
 En appliquant le lemme~\ref{lem:tangence} aux homographies $\g^{\pm1}$, il apparaît que les cercles $I(\g^{\pm 1})$ sont de rayon $1$ et centrés en $0$ et $b$. Ces cercles étant disjoints, nous avons $|b|\geq 2$. Par minimalité de $|b|$, nous trouvons $|\omega|\geq 2|b|\geq 4$.\par
  Pour identifier le cas d'égalité, nous allons légèrement modifier la situation. Nous posons $B_\infty=\{z\in\C~;~\Imaginary(z)>1/2\}$ et $\G_\infty=\langle z\mapsto z+2 \rangle$. Les cercles $I(\g^{\pm 1})$ sont alors de rayon $1/2$ et tangents. Disons que $I(\g^{\pm 1})$ est le cercle de rayon $1/2$ centré en $\pm 1/2$. Dans ce cas, nous avons explicitement $\g:z\mapsto -z/(2z-1)$. On conclut en remarquant que $z\mapsto z+2$ et $\g$ forment un base du sous-groupe de congruence modulo $2$ de $\PSL(2,\Z)$, dont le quotient est une sphère à trois pointes.
\end{proof}

Le fait suivant est bien connu~:
\begin{corollary}\label{cor:simple}
Aucune géodésique fermée simple de $X$ n'entre dans la région cuspidale image de l'horodisque $\{z\in\Hyp~; \Imaginary(z)>\sqrt{1+\omega^2/4}~\}$. Ainsi, aucune géodésique fermée simple d'une surface hyperbolique orientable n'entre dans une région cuspidale d'aire $2$.
\end{corollary}

\begin{proof}
Tout relevé d'une géodésique fermée simple est disjoint de son image par $z\mapsto z+\omega$.
\end{proof}

\subsubsection{Un lemme} Nous nous plaçons dans la situation du paragraphe précédent. Soit $z_0\in\Hyp$ un point quelconque de partie imaginaire $\Imaginary(z_0)\geq\sqrt{3}$. 

\begin{lemma}\label{lem:minoration-distance}
Si $\g$ est un élément hyperbolique de $\G$ dont l'axe pénètre $B_\infty$, alors
$$\sinh\left(\frac{d_\Hyp(z_0,\g\cdot z_0)}{2}\right)> \frac{2}{\sqrt{3}}.$$
\end{lemma}

\begin{proof}
Nous notons $d$ la distance de translation de $\g$ le long de son axe, et $h$ la distance hyperbolique entre $z_0$ et l'axe de $\g$. La distance de déplacement de $z_0$ par $\g$ est donnée par la formule~: 
 \begin{eqnarray}\label{eq:deplacement}
 \sinh(d_\Hyp(z_0,\g(z_0))/2) & = & \sinh(d/2)\cosh(h).
 \end{eqnarray}\par

Nous supposons que l'axe de $\g$ est le cercle de rayon $r>0$ centré en l'origine. Dans ce cas, nous avons $h\geq \ln(\sqrt{3}/r)$ soit $\cosh(h)\geq(3+r^2)/2\sqrt{3}r$. En utilisant les égalités de la fin du \textsection~\ref{sec:cercles}, nous trouvons
$$\sinh\left(\frac{d_\Hyp(z_0,\g\cdot z_0)}{2}\right) \geq \frac{(3+r^2)}{2\sqrt{3}} \frac{\tan\theta}{r} = \frac{(3+r^2)}{2\sqrt{3}R}.$$
Il ne reste plus qu'à injecter les inégalités $R\leq 1$ ($B_\infty$ se projette sur une région cuspidale) et $r> 1$ (la géodésique $\g$ de $X$ pénètre cette région cuspidale).
\end{proof}

\end{document}